\documentclass[12pt]{amsart}

\usepackage{amsthm,amsmath,amssymb}
\RequirePackage{hyperref}
\usepackage[utf8]{inputenc} 

\usepackage{textcomp}

\usepackage{graphicx}

\usepackage{a4wide}

\DeclareMathOperator{\Forall}{\forall\,}

\newcommand{\dx}{\,\mathrm{d}x}

\newcommand{\bfn}{\boldsymbol n}

\newcommand{\bfx}{\boldsymbol x}

\newcommand{\bfv}{\boldsymbol v}
\newcommand{\bfu}{\boldsymbol u}
\newcommand{\bfw}{\boldsymbol w}

\newcommand{\bff}{\boldsymbol f}

\newcommand{\bfzero}{\boldsymbol 0}

\newcommand{\bfV}{\boldsymbol V\!} 
\newcommand{\bfW}{\boldsymbol W\!} 
\newcommand{\bfpi}{{\boldsymbol \pi}}

\newcommand{\mcK}{\mathcal{K}}

\newcommand{\tn}{|\mspace{-1mu}|\mspace{-1mu}|}

\DeclareMathOperator*{\divv}{div}
\newcommand{\IR}{\mathbb{R}}

\newcommand{\msE}{\mathsf{E}}

\newtheorem{theorem}{Theorem}
\newtheorem{thm}[theorem]{Theorem}
\newtheorem{prop}[theorem]{Proposition}
\newtheorem{rem}[theorem]{Remark}
\newtheorem{lem}[theorem]{Lemma}

\numberwithin{equation}{section}  

\newtheorem{ass}{Assumption}

\newcommand\csentence[1]{#1}  

\begin{document}




\title{High Order Cut Finite Element Methods for the Stokes Problem}

\author{August Johansson}
\email{august@simula.no}
\address{Center for Biomedical Computing, Simula Research Laboratory, Norway}

\author{Mats G.\ Larson}
\email{mats.larson@math.umu.se}
\address{Department of Mathematics, Ume{\aa} University, Sweden}

\author{Anders Logg}
\email{logg@chalmers.se}
\address{Mathematical Sciences, Chalmers University of Technology and University of Gothenburg, Sweden}

\thanks{This research was supported in part by the Swedish Foundation
  for Strategic Research Grant No.\ AM13-0029, the Swedish Research
  Council Grant No.\ 2013-4708, and the Swedish Research Council Grant
  No.\ 2014-6093.  The work was also supported by The Research Council
  of Norway through a Centres of Excellence grant to the Center for
  Biomedical Computing at Simula Research Laboratory, project number
  179578. }

 \begin{abstract} 
  We develop a high order cut finite element method for the Stokes
  problem based on general inf-sup stable finite element spaces. We
  focus in particular on composite meshes consisting of one mesh that
  overlaps another. The method is based on a Nitsche formulation of
  the interface condition together with a stabilization term. Starting
  from inf-sup stable spaces on the two meshes, we prove that the
  resulting composite method is indeed inf-sup stable and as a
  consequence optimal \emph{a~priori} error estimates hold.


\end{abstract}

\maketitle 

\section{Background}

\subsection{Introduction}

Meshing of complex geometries remains a challenging and time consuming
task in engineering applications of the finite element method. There
is therefore a demand for finite element methods based on more
flexible mesh constructions. One such flexible mesh paradigm is the
formulation of finite element methods on composite meshes created by
letting several meshes overlap each other. This approach enables using
combinations of meshes for certain parts of a domain and reuse of
meshes for complicated parts that may have been difficult and time
consuming to construct.

We consider the case of a composite mesh consisting of one mesh that
overlaps another mesh which together provide a mesh of the
computational domain of interest. This results in some elements on one
mesh having an intersection with one or several elements on the
boundary of the other mesh. We denote such elements by cut
elements. The interface conditions on these cut elements are enforced
weakly and consistently using Nitsche's method \cite{Nitsche71}.

In this setting \cite{HanHanLar} first developed and analyzed a
composite mesh method for elliptic second order problem based on
Nitsche's method. In \cite{MassingStokes}, this approach was extended
to the Stokes problem using suitable stabilization to ensure inf-sup
stability of the method. Implementation aspects were discussed in
detail in \cite{MasLarLog13}. In \cite{ZahediP1isoP2} a related cut
finite element method for a Stokes interface problem based on the
P1-iso-P2 element was developed and analyzed.

Composite mesh techniques using domain decomposition are often called
chimera, see for example \cite{henshaw}, \cite{chimera} for uses in a
finite difference setting or \cite{codina} in a finite element
setting. The extended finite element method (XFEM) also provides
composite mesh handling techniques, see for example
\cite{wall1,wall2}. However, the Nitsche method approach using cut
elements used in this work makes it possible to obtain a consistent
and stable formulation while maintaining the conditioning of the
algebraic system for both conforming and non-conforming high order
finite elements.

In this paper, we consider Stokes flow and device a method based on a
stabilized Nitsche formulation for enforcement of the interface
conditions at the border between the two meshes. A specific feature is
that we only assume that we have inf-sup stable spaces on the two
meshes and that the spaces consist of polynomials. We can then show
that our stabilized Nitsche formulation satisfies an inf-sup condition
and as a consequence optimal order also \emph{a~priori} error
estimates hold. We emphasize that the spaces are arbitrary and can be
different on the two meshes, in particular, continuous or
discontinuous pressure spaces as well as higher order spaces can be
used. We present extensive numerical results for higher order
Taylor-Hood elements in two and three spatial dimensions that confirm
our theoretical results.


The outline of the paper is as follows: First we review the Stokes
problem. Then the finite element method is presented by first defining
the composite mesh and introducing finite element spaces. The method
is then analyzed where the inf-sup condition is the main
result. Finally we present the numerical results and the conclusions.

\subsection{The Stokes Problem}

In this section, we review the Stokes problem and state its
standard weak formulation. We also introduce some basic notation.

\subsubsection{Strong form}

Let $\Omega$ be a polygonal domain in $\IR^d$ with boundary
$\partial \Omega$.  The Stokes problem takes the form: Find the
velocity $\bfu : \Omega \rightarrow \IR^d$ and pressure
$p : \Omega \rightarrow \IR$ such that
\begin{alignat}{2}
  \label{eq:stokes1}
  -\Delta \bfu + \nabla p &= \bff \qquad &&\text{in $\Omega$},
  \\
  \label{eq:stokes2}
  \divv \bfu &= 0 \qquad &&\text{in $\Omega$},
  \\
  \bfu &= \bfzero \qquad &&\text{on $\partial \Omega$},
\end{alignat}
where $\bff: \Omega \rightarrow \IR^d$ is a given right-hand side.

\subsubsection{Weak form}

As usual, let $H^s(\Omega)$ denote the standard Sobolev space of order
$s\geq 0$ on $\Omega$ with norm denoted by $\|\cdot \|_{H^s(\Omega)}$ and
semi-norm denoted by $|\cdot |_{H^s(\Omega)}$. Let $L^2(\Omega)$ denote the
$L^2$-norm on $\Omega$ with norm denoted by $\| \cdot \|_{\Omega}$. The
corresponding inner products are labeled accordingly.

Introducing the spaces
\begin{align}
  \bfV &= [H^1_0(\Omega)]^d, \\
  Q &= \{ q \in L^2(\Omega) : \int_\Omega q \dx = 0 \},
\end{align}
with norms $\|D \bfv \|_{\Omega} = \| \bfv \otimes \nabla \|_{\Omega}$
and $\| q \|_\Omega$, the weak form of \eqref{eq:stokes1} and
\eqref{eq:stokes2} reads: Find $(\bfu,p) \in \bfV \times Q$ such that
\begin{equation}\label{eq:weakform}
  a(\bfu,\bfv) + b(\bfu,q) + b (p,\bfv) = l(\bfv)\quad \Forall
  (\bfv,q) \in \bfV \times Q,
\end{equation}
where the forms are defined by
\begin{align}
  a(\bfu,\bfv) &= (D \bfu, D \bfv)_\Omega,
  \\
  b(\bfu,q) &= -(\divv \bfu,q)_\Omega = (\bfu, \nabla q)_\Omega,
  \\
  l(\bfv) &= (\bff,\bfv)_\Omega.
\end{align}
\begin{rem}
  We obtain the variational problem~\eqref{eq:weakform} by formally
  multiplying~\eqref{eq:stokes1} by a test function $\bfv$
  and~\eqref{eq:stokes2} by a test function $-q$.
\end{rem}
It is then possible to show that the inf-sup condition
\begin{equation}
  \|q\|_\Omega \lesssim \sup_{\bfv \in \bfV} \frac{b(\bfv,q)}{\| D\bfv \|_\Omega}
  = \sup_{\bfv\in V} \frac{(\divv \bfv, q)}{\|D \bfv\|_{\Omega}}
  \quad \Forall q \in Q
\end{equation}
holds, from which it follows that there exists a unique solution to
(\ref{eq:weakform}). See \cite{BreFor} for further details.

\section{Methods}

\subsection{The Composite Mesh}
\label{sect:meshstuff}

We here present the concepts and notation of the domains and meshes
used. The main idea is to introduce a background domain which is
partially overlapped by another domain (the overlapping domain).  For
each of these domains, we mimic the setup of a traditional finite
element method in the sense that each domain is equipped with a
traditional finite element mesh. The two meshes are completely
unrelated. In particular, the interface between the two meshes is
determined by the overlapping domain and is not required to match or
align with the triangulation of the background domain.

\subsubsection{The composite domain}

Let the \emph{predomains} ${\widehat{\Omega}}_i \subset \Omega$,
$i=0,1$, be polygonal subdomains of $\Omega$ in $\IR^d$ such that
$\widehat{\Omega}_0 \cup \widehat{\Omega}_1 = \Omega$; see Figure
\ref{fig:domains}.  Consider the partition
\begin{align}
  \label{eq:omega_partition}
  \Omega &= \Omega_0 \cup \Omega_1, \\
  \Omega_0 &= \Omega \setminus \widehat{\Omega}_1, \\
  \Omega_1 &= \widehat{\Omega}_1,
\end{align}
and let $\Gamma = \partial \Omega_1 \setminus \partial \Omega$ be the
interface between the overlapping domain $\Omega_1$ and the underlying
domain $\Omega_0$; see Figure \ref{fig:domains}. We make the basic
assumption that each $\Omega_i$, $i = 0,1$, has a nonempty interior.
We note that implies that there exists a nonempty open set
$U\in\Omega$ such that $\Gamma\cap U \neq \emptyset$. (The set $U$
plays an important role in the proof of
Lemma~\ref{lem:infsupconstants} below.)

\begin{figure}
  \includegraphics[width=0.65\textwidth]{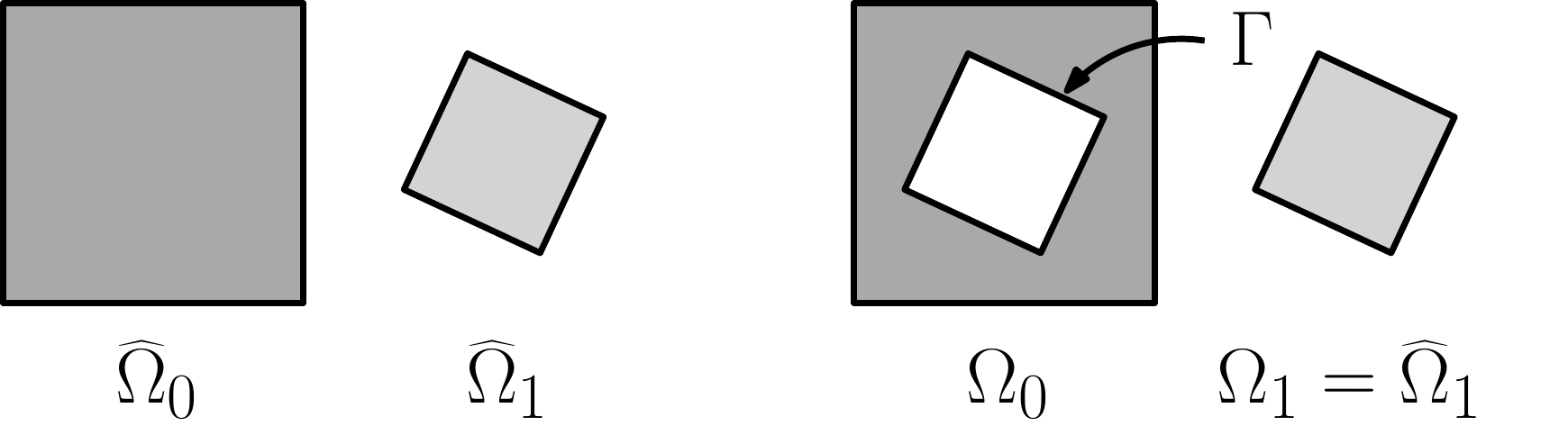}
  \caption{\csentence{The domains $\widehat{\Omega}_i$ and the subdomains
    $\Omega_i$ (all shaded) sharing the interface $\Gamma$.}}
  \label{fig:domains}
\end{figure}

\subsubsection{The composite mesh}

For $i=0,1$, let $\widehat{\mcK}_{h,i}$ be a quasi-uniform mesh on
$\widehat{\Omega}_i$ with mesh parameter $h \in (0,\bar{h}]$ and let
\begin{equation}
  \mcK_{h,i} = \{K \in \widehat{\mcK}_{h,i} :
  \overline{K} \cap \Omega_i \neq \emptyset \}
\end{equation}
be the submesh consisting of elements that intersect $\Omega_i$; see
Figure \ref{fig:meshes}. Note that $\mcK_{h,0}$ includes elements that
partially intersect $\Omega_1$. We also introduce the notation
\begin{equation}
  \label{eq:omegah}
  \Omega_{h,i} = \bigcup_{K \in \mcK_{h,i}} K.
\end{equation}
Note that $\Omega_1 = \Omega_{h,1}$ and
$\Omega_0 \subset \Omega_{h,0}$; see Figure \ref{fig:Omega_h}.

\begin{figure}
 \includegraphics[width=0.3\textwidth]{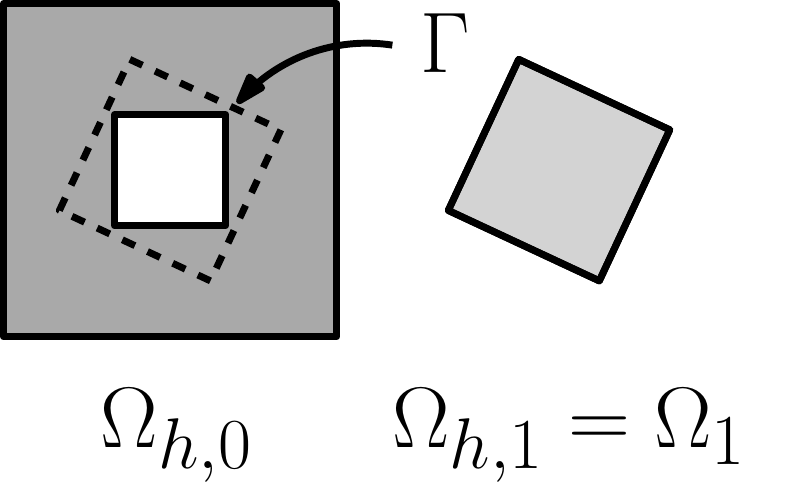}
  \caption{\csentence{The domains $\Omega_{h,i}$ (shaded).}}
  \label{fig:Omega_h}
\end{figure}

We obtain a partition of $\Omega$ by intersecting the elements with
the subdomains:
\begin{equation}
  \bigcup_{i=0}^1 \mcK_{h,i} \cap \Omega_i = \bigcup_{i=0}^1
  \{K \cap \Omega_i : K \in \mcK_{h,i} \}.
\end{equation}
See also Figure \ref{fig:meshes}.

\begin{figure}
  \includegraphics[width=0.65\textwidth]{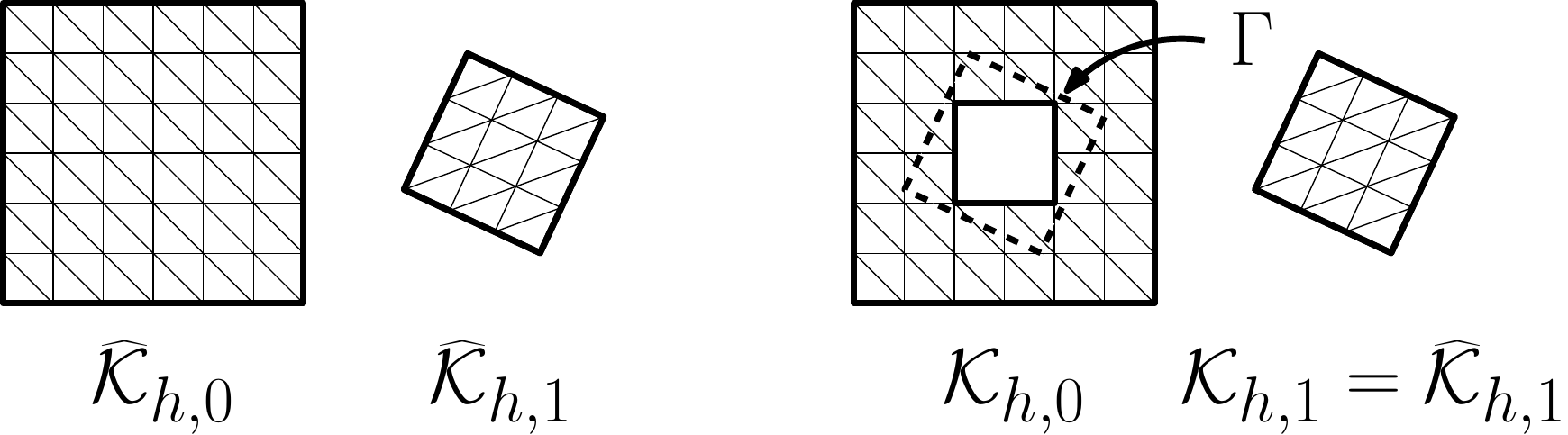}
  \caption{\csentence{The meshes $\widehat{\mcK}_{h,i}$ and $\mcK_{h,i}$ of the
    corresponding domains $\widehat{\Omega}_i$ and
    $\Omega_{h,i}$.} Note that $\Gamma$ is not aligned with
    $\mcK_{h,0}$.}
  \label{fig:meshes}
\end{figure}


\subsection{Finite element formulation}

In this section, we present the finite element method for
approximating the weak form \eqref{eq:weakform}. Some notation will be
introduced, but the main idea is to assume we have inf-sup stable
spaces in each of the subdomains away from the interface. Then we are
able to formulate a method similar to \cite{HanHanLar} and
\cite{MassingStokes}.

\subsubsection{Finite element spaces}

For each of the predomains $\widehat{\Omega}_i$ with corresponding
family of meshes $\widehat{\mcK}_{h,i}$ we consider velocity and
pressure finite element spaces
$\widehat{\bfV}_{h,i}\times \widehat{Q}_{h,i}$. The spaces do not
contain boundary conditions since these will be enforced by the finite
element formulation. We define
\begin{equation}
  \bfV_{h,i} \times Q_{h,i}
  = \widehat{\bfV}_{h,i}|_{\Omega_{h,i}}
  \times \widehat{Q}_{h,i}|_{\Omega_{h,i}},
\end{equation}
where $i=0,1$ and define
\begin{equation}
  \bfV_h \times Q_h
  = \bigoplus_{i=0}^1 \bfV_{h,i} \times Q_{h,i}.
\end{equation}
Note that since the domains $\Omega_{h,i}$ overlap each other,
$\bfV_h \times Q_h$ is to be understood as a collection of function
spaces on the overlapping patches $\Omega_{h,i}$, $i=0,1$. We now make
the following fundamental assumptions on these spaces:


\begin{ass}[Piecewise polynomial spaces]
  \label{assumptionA}
  The finite element spaces $\bfV_h$ and $Q_h$ consist of piecewise
  polynomials of uniformly bounded degree $k$ and $l$, respectively.
\end{ass}

\begin{ass}[Inf-sup stability]
  \label{assumptionB}
  The finite element spaces are inf-sup stable restricted to a domain
  bounded away from the interface. More precisely, we assume that for
  $i=0,1$ and $h \in (0,\bar{h}]$ there is a domain
  $\omega_{h,i} \subset \Omega_i$ such that:
    \begin{trivlist}
    \item(a) The set $\omega_{h,i}$ is a union of elements
      in $\mcK_{h,i}$; see Figure \ref{fig:small_omega_h}.
    \item(b) The inf-sup condition
      \begin{equation}\label{assum:infsup}
        m_i \| p_i - \lambda_{\omega_{h,i}}(p) \|_{\omega_{h,i}} \leq \sup_{\bfv \in \bfW_{h,i}} \frac{(\divv \bfv,p)_{\omega_{h,i}}}{\|D \bfv \|_{\omega_{h,i}}}
      \end{equation}
      holds, where $\lambda_{\omega_{h,i}}(p)$ is the average of $p$
      over $\omega_{h,i}$ and $\bfW_{h,i}$ is the subspace
      of $\bfV_{h,i}$ defined by
      \begin{align}
        \label{eq:bfW}
        \bfW_{h,0} &=
        \{\bfv \in \bfV_{h,0}: \text{$\bfv=\bfzero$ \text{on} $\overline{\Omega_{h,0}
          \setminus \omega_{h,0}}$}\}, \\
        \bfW_{h,1} &= \bfV_{h,1}.
      \end{align}
    \item(c) The set $\omega_{h,0}$ is close to $\Omega_0$
      in the sense that
      \begin{equation}\label{assum:omegahclose}
        \Omega_{h,0} \setminus \omega_{h,0}
        \subset U_\delta(\Gamma),
        \quad \delta \sim h,
      \end{equation}
      where $U_\delta(\Gamma) =\{ \bfx \in \IR^d :
      |\rho(\bfx)|<\delta\}$ is the tubular neighborhood of
      $\Gamma$ with thickness $\delta$.
    \end{trivlist}
\end{ass}

\begin{rem}
  The assumptions presented ensure that the polynomial spaces are such
  that certain inverse inequalities hold. More generally, inverse
  inequalities hold if there is a finite set of finite dimensional
  reference spaces used to construct the element spaces. The use of
  the interpolant in the proof of Lemma \ref{lem:infsupconstants}
  could alternatively be handled using an abstract approximation
  property assumption.
\end{rem}

\begin{figure}
  \includegraphics[width=0.3\textwidth]{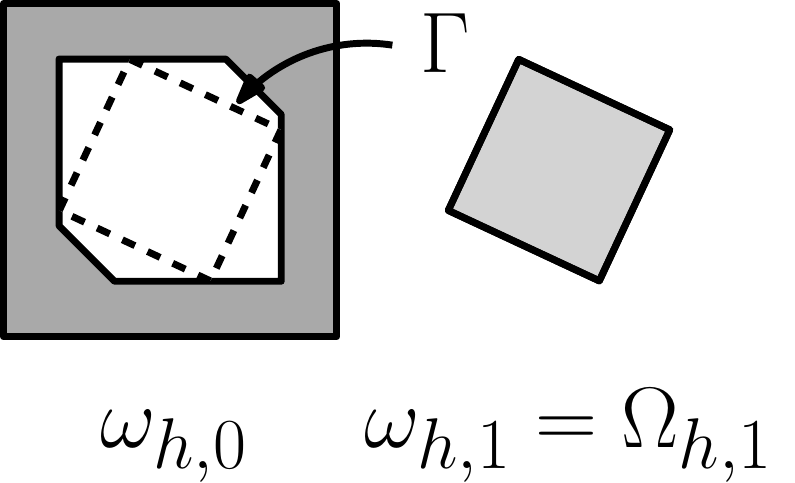}\qquad
  \caption{\csentence{The domains $\omega_{h,i}$ (shaded) where inf-sup stability
    is assumed.}  Note that $\Gamma$ is outside $\omega_{h,0}$.}
  \label{fig:small_omega_h}
\end{figure}


\subsubsection{Finite element method}

We consider the finite element method: Find
$(\bfu_h,p_h) \in \bfV_h \times Q_h$ such that
\begin{equation}
  \label{eq:fem}
  A_h((\bfu_h,p_h),(\bfv,q)) = l_h(\bfv) \quad \Forall
  (\bfv,q) \in \bfV_h \times Q_h,
\end{equation}
where the forms are defined by
\begin{align}
  \label{eq:Ah}
  A_h((\bfu,p),(\bfv,q)) &= a_h(\bfu,\bfv) + b_h(\bfu,q)
  + b_h(\bfv,p) + d_h((\bfu,p),(\bfv,q)), \\
  \label{eq:ah}
  a_h(\bfu,\bfv) &= a_{h,N}(\bfu,\bfv) + a_{h,O}(\bfu,\bfv),
  \\
  a_{h,N}(\bfu,\bfv) &= (D \bfu, D \bfv)_{\Omega_0}
  + (D \bfu, D \bfv)_{\Omega_1}
  \\ \nonumber
  &\qquad -
  (\langle (D \bfu) \cdot \bfn \rangle,[ \bfv])_{\Gamma}
  - ([ \bfu], \langle (D \bfv) \cdot \bfn \rangle)_{\Gamma}
  \\ \nonumber
  &\qquad + \beta h^{-1}([\bfu],[ \bfv])_{\Gamma},
  \\
  a_{h,O}(\bfu,\bfv)&=
  ([D\bfu], [D\bfv])_{\Omega_{h,0} \cap \Omega_1},
  \\
  \label{eq:bh}
  b_h(\bfu,q) &= - (\divv \bfu,q)_{\Omega_0}
  - (\divv \bfu,q)_{\Omega_1}
  + ([\bfn \cdot \bfu],\langle q \rangle )_{\Gamma},
  \\
  \label{eq:dh}
  d_h((\bfu,p),(\bfv,q))&=
  h^2(\Delta \bfu - \nabla p,\Delta \bfv + \nabla q)_{\Omega_{h,0}\setminus \omega_{h,0}},
  \\
  l_h(\bfv) &= (\bff,\bfv)_\Omega
  - h^2(\bff,\Delta \bfv + \nabla q)_{\Omega_{h,0}\setminus \omega_{h,0}}.
\end{align}
Here, $\bfn$ is the unit normal to $\Gamma$ exterior to $\Omega_1$,
$[v] = v_1 - v_0$ is the jump at the interface $\Gamma$ and
$\langle v \rangle = (v_0 + v_1)/2$ is the average at $\Gamma$
(although any convex combination is valid \cite{HanHanLar}). The
parameter $\beta>0$ is the Nitsche parameter and must be sufficiently
large (see for example \cite{HanHanLar}) and scales as $k^2$, where
$k$ is the polynomial degree. Furthermore, $h$ is the representative
mesh size of the quasi-uniform mesh. In a practical implementation,
$h$ is evaluated as the local element size.

A comment on the respective terms may be clarifying: $a_{h,N}$ and
$b_h$ are the standard Nitsche formulation of \eqref{eq:weakform} and
$a_{h,O}$ is a stabilization of the jump of the gradients across
$\Gamma$ (see \cite{MassingStokes}). The least-squares type term $d_h$
stabilizes the method since we do not assume inf-sup stability in all
of $\Omega_0$.

By simple inspection, we note that the method is consistent.  We
conclude by noting that the method satisfies the Galerkin
orthogonality.

\begin{prop}[Galerkin orthogonality]
  \label{prop:galerkinorth}
  Let $(\bfu,p) \in \bfV \times Q$ be a weak solution to the
  formulation \eqref{eq:weakform} and let $(\bfu_h,p_h) \in \bfV_h
  \times Q_h$ be the solution to the finite element formulation
  \eqref{eq:fem}. Then it holds
  \begin{align}
    \label{eq:galerkinorth}
    A_h((\bfu,p)-(\bfu_h,p_h),(\bfv_h,q_h)) = 0 \quad \Forall
    (\bfv_h,q_h) \in \bfV_h \times Q_h.
  \end{align}
\end{prop}
\begin{proof}
  The result follows from \cite{HanHanLar} and noting that
  $a_{h,O}(\bfu,\bfv_h)=d_h((\bfu,p),(\bfv_h,q_h))=0$ for all
  $(\bfv_h,q_h) \in \bfV_h \times Q_h$.
\end{proof}

\subsubsection{Approximation properties}

We assume that there is an interpolation operator
$\bfpi_{h,i}: \bfV_{i}(\Omega_{h,i}) \rightarrow \bfV_{h,i},$ for
$i=0,1$, where $\bfV_{i}(\Omega_{h,i}) \subset [L^2(\Omega_{h,i})]^d$
is a space of sufficient regularity to define the interpolant. For
Taylor-Hood elements, we take $\bfpi_{h,i}$ to be the Scott-Zhang
interpolation operator~\cite{ScottZhang}, and
$\bfV_{i}(\Omega_{h,i}) = [L^2(\Omega_{h,i})]^d$. For other elements
we refer to their corresponding papers, for example the
Crouzeix-Raviart element \cite{CR1973}, the Mini element
\cite{mini}, or the overviews in \cite{boffi} or \cite{braess}.

The full interpolation operator $\bfpi_h: \bfV \rightarrow \bfV_h$ can
now be defined by the use of a linear extension operator
$\msE: [H^s(\omega_{h,0})]^d \rightarrow [H^s(\Omega_{h,0})]^d$,
$s\geq 0$, such that $(\msE \bfv)|_{\omega_{h,0}} = \bfv$ and
\begin{align}
  \| \msE \bfv \|_{H^s(\Omega_{h,0})} \lesssim \| \bfv \|_{H^s(\omega_{h,0})}.
\end{align}
Now, $\bfpi_h: \bfV \rightarrow \bfV_h$ is defined by
\begin{align}
  \bfpi_h \bfv = \bfpi_{h,0} \msE \bfv_0  \oplus \bfpi_{h,1} \bfv_1.
\end{align}
A similar argument can be made to define the pressure interpolation
operator $\pi_h: Q \rightarrow Q_h$.

Furthermore, we assume that the following standard interpolation
estimate holds:
\begin{equation}\label{eq:interpol}
  \|v - \pi_h v\|_{H^m(K)} \lesssim h^{k+1-m} | v |_{H^{k+1}(\widetilde{K})},
  \quad m=0,1, \ldots, k.
\end{equation}
Here, in the case of a Scott-Zhang interpolation operator,
$\widetilde{K}$ is the patch of elements neighboring $K$.

\subsection{Stability and Convergence}

In this section, we prove that the finite element method proposed in
\eqref{eq:fem} is stable. This is done by first proving the coercivity
and continuity of $a_h$ defined in~\eqref{eq:ah}, followed by proving
that $b_h$ defined in~\eqref{eq:bh} satisfies the inf-sup
condition. Combining these results proves stability of $A_h$. This
strategy is similar to what can be found in \cite{MassingStokes} and
\cite{ZahediP1isoP2}. In particular, Verf\"urth's trick
\cite{verfurth} is used to prove inf-sup stability.  For a general
overview of the saddle point theory used, see \cite{boffi, braess,
  BreFor}. We conclude the section by proving an \emph{a~priori} error
estimate. Before we begin, we state appropriate norms.

\subsubsection{Norms}
In the analysis that follows, we shall use the following norms:
\begin{align}
  \label{eq:energynorm}
  \tn \bfv \tn^2_h &=
  \sum_{i=0}^1 \| D \bfv_i \|^2_{\Omega_{h,i}}
  +  h \| \langle D \bfv \rangle \cdot \bfn \|^2_{\Gamma}
  +  h^{-1} \| [\bfv] \|^2_{\Gamma},
  \quad \bfv \in \bfV_h,
  \\
  \|q\|^2_h &= \sum_{i=0}^1 \| q_i \|^2_{\Omega_{h,i}},
  \quad q \in Q_h,
  \\
  \label{eq:triplenorm}
  \tn (\bfv,q) \tn^2_h &= \tn \bfv \tn^2_h + \| q \|^2_h,
  \quad (\bfv,q) \in \bfV_h \times Q_h.
\end{align}

\subsubsection{Interpolation estimates}

Using \eqref{eq:interpol} together with the trace inequality
$\| v \|^2_{\Gamma \cap K} \lesssim h^{-1} \| v \|^2_K + h \| \nabla v
\|^2_K$,
we obtain the following interpolation estimate for $\bfv \in \bfV$:
\begin{align}
  \label{eq:interpolu}
  \tn \bfv - \bfpi_h \bfv \tn_h &\lesssim h^k | \bfv |_{H^{k+1}(\Omega)}.
\end{align}
See \cite{HanHanLar} for a proof. For the pressure $p$, we have
\begin{equation}
  \|p - \pi_h p\|_h \lesssim h^{l+1} | \bfv |_{H^{l+1}(\Omega)}.
\end{equation}

\subsubsection{Coercivity and continuity}

Establishing coercivity and continuity of $a_h$ is straightforward and
similar to \cite{HanHanLar}.


\begin{lem}[Coercivity of $a_h$]
  \label{lem:ahcoer}
  The bilinear form $a_h$ \eqref{eq:ah} is coercive:
  \begin{equation}\label{eq:ahcoercivity}
    \tn \bfv \tn_h ^2 \lesssim a_h(\bfv,\bfv) \quad
    \Forall \bfv \in \bfV_h.
  \end{equation}
\end{lem}
\begin{proof}
  Note that the overlap term $a_{h,O}$ provides the control
  \begin{align}
    \sum_{i=0}^1 \|D \bfv \|^2_{\Omega_{h,i}}
    &=
    \| D \bfv_0 \|^2_{\Omega_{h,0}\setminus \Omega_1}
    +  \| D \bfv_0 \|^2_{\Omega_{h,0}\cap \Omega_1}  + \| D \bfv_1 \|^2_{\Omega_{h,1}}
    \\
    &\lesssim
    \| D \bfv_0 \|^2_{\Omega_{h,0} \setminus \Omega_1}
    +  \| D (\bfv_0-\bfv_1)\|^2_{\Omega_{h,0}\cap \Omega_1}
    + \| D \bfv_1 \|^2_{\Omega_1}
    \\
    &\leq \sum_{i=0}^1 \|D \bfv \|^2_{\Omega_{i}} + a_{h,O}(\bfv,\bfv),
  \end{align}
  where we have used that $\Omega_{h,0} \setminus \Omega_1 = \Omega_0$
  and $\Omega_{h,0}\cap \Omega_1\subset \Omega_1$ as described in the
  section on the composite mesh above. We also note that for each
  element $K$ that intersects an interface segment $\Gamma$ we have
  the inverse bound
  \begin{equation}\label{eq:inverseedge}
    h\| (D \bfv)\cdot \bfn \|_{K\cap \Gamma}^2 \lesssim
    \| D \bfv \|_K^2
  \end{equation}
  independent of the particular position of the intersection between
  $K$ and $\Gamma$ (see \cite{HanHanLar}). Combining these two
  estimates with the standard approach to establish coercivity of a
  Nitsche method (see for example \cite{HanHanLar}) immediately gives
  the desired estimate.
\end{proof}

\begin{lem}[Continuity of $a_h$]
  \label{lem:ahcont}
  The bilinear form $a_h$ \eqref{eq:ah} is continuous:
  \begin{equation}\label{eq:ahcont}
    a_h(\bfv,\bfw) \lesssim \tn \bfv \tn_h \tn \bfw \tn_h
    \quad
    \Forall \bfv,\bfw \in \bfV_h.
  \end{equation}
\end{lem}
\begin{proof}
  A proof in absence of $a_{h,O}$ is found in
  \cite{HanHanLar}. Bounding $a_{h,O}$ is straightforward using the
  Cauchy-Schwarz inequality and the fact that $\| D \bfw \| \lesssim
  \tn \bfw \tn_h$ for any $\bfw \in \bfV_h$.
\end{proof}

\subsubsection{Stability}

Showing stability of the proposed finite element method involves
several steps. First we show a preliminary stability estimate for
$A_h$ \eqref{eq:fem}. Then the so called small inf-sup condition for
$b_h$ \eqref{eq:bh} is shown using a decomposition of the pressure
space into $L^2$ orthogonal components. For each of these components
we show that an inf-sup condition holds. This is then used to show the
big inf-sup condition for $A_h$.

\begin{lem}[Preliminary stability estimate for $A_h$]
  \label{lem:infsuptestu}
  It holds
  \begin{align}
    \tn \bfu \tn^2_h + h^2
    \|\nabla p\|_{\Omega_{h,0}\setminus \omega_{h,0}}^2
    &\lesssim
    \tn \bfu \tn^2_h
    + h^2 \|\Delta \bfu - \nabla p\|_{\Omega_{h,0}\setminus \omega_{h,0}}^2\\
    &\lesssim A_h((\bfu,p),(\bfu,-p)).
  \end{align}
\end{lem}
\begin{proof}
  Recall the inverse estimate
  \begin{align}
    \label{eq:inverse}
    \| \bfv \|_{H^l(K)} \leq C h^{m-l} \| \bfv \|_{H^m(K)}
  \end{align}
  (see \cite{BreSco08}, Section 4.5) where $K\in \mcK_{h,i}$ and
  $\bfv\in \bfV_h$.  The first estimate in the lemma follows by adding
  and subtracting $\Delta \bfu$, using the triangle inequality and
  \eqref{eq:inverse} as follows:
  \begin{align}
    h^2 \| \nabla p \|^2_K
    &\leq
    h^2\| \nabla p - \Delta \bfu \|^2_K
    +
    h^2\| \Delta \bfu \|^2_K
    \\
    &\lesssim
    h^2\| \nabla p - \Delta \bfu \|^2_K
    +
    \| D \bfu \|^2_K,
  \end{align}
  for each element $K\in \mcK_{h,i}$. The second estimate follows
  immediately using coercivity \eqref{eq:ahcoercivity} since
  \begin{align}
    A_h((\bfu,p),(\bfu,-p))
    &=
    a_h (\bfu,\bfu)
    +
    d_h ((\bfu,p),(\bfu,-p))
    \\
    &= a_h (\bfu,\bfu)  + h^2\| \nabla p - \Delta \bfu \|^2_{\Omega_{h,0}\setminus \omega_{h,0}}
    \\
    &\gtrsim \tn \bfu \tn_h^2 + h^2\| \nabla p - \Delta \bfu \|^2_{\Omega_{h,0}\setminus \omega_{h,0}}.
  \end{align}
\end{proof}

The pressure space can be written as the following $L^2$-orthogonal
decomposition:
\begin{equation}
  Q = Q_c \oplus Q_0 \oplus Q_1,
\end{equation}
where $Q_c$ is the space of piecewise constant functions on the
partition $\{\Omega_i\}_{i=0}^1$ of $\Omega$ with average zero over
$\Omega$ and $Q_i$ is the space of $L^2$ functions with average zero
over $\Omega_i$. We next show inf-sup conditions for $Q_c$ and
$Q_0$. Recall that the inf-sup condition for $Q_1$ is already
established by Assumption~\ref{assumptionB}.

\begin{lem}[Inf-sup for $Q_c$] \label{lem:infsupconstants}
  For each $q \in Q_c$ there exists a $\bfw_c \in \bfV_h$ with
  $\tn \bfw_c \tn_h = \|q \|_h$ such that
  \begin{equation}
    \| q \|_h^2 \lesssim b_h(\bfw_c,q),
  \end{equation}
  where the bound is uniform w.r.t.~$q$.
\end{lem}
\begin{proof}
  We first note that $Q_c$ is a one-dimensional vector space spanned
  by
  \begin{equation}
    \label{eq:chi}
    \chi =
    \begin{cases}
      |\Omega_0|^{-1} & \text{in $\Omega_0$},
      \\
      -|\Omega_1|^{-1} &
      \text{in $\Omega_1$}.
    \end{cases}
  \end{equation}
  Second, we note that since $\Omega_0$ and $\Omega_1$ are
  nonempty, there exists a nonempty open set $U \subset \Omega$ such
  that $\Gamma \cap U \neq \emptyset$.  Let now
  $\bfx_0 \in \Gamma \cap U$ be a point on the interface $\Gamma$ and
  let $B_R(\bfx_0)$ be a ball of radius $R$ centered at $x_0$ as in
  Figure~\ref{fig:ball}. The radius $R$ is chosen such that
  $B_R(\bfx_0) \subset U$ independently of the mesh size $h$.
  \begin{figure}
    \includegraphics[width=0.3\textwidth]{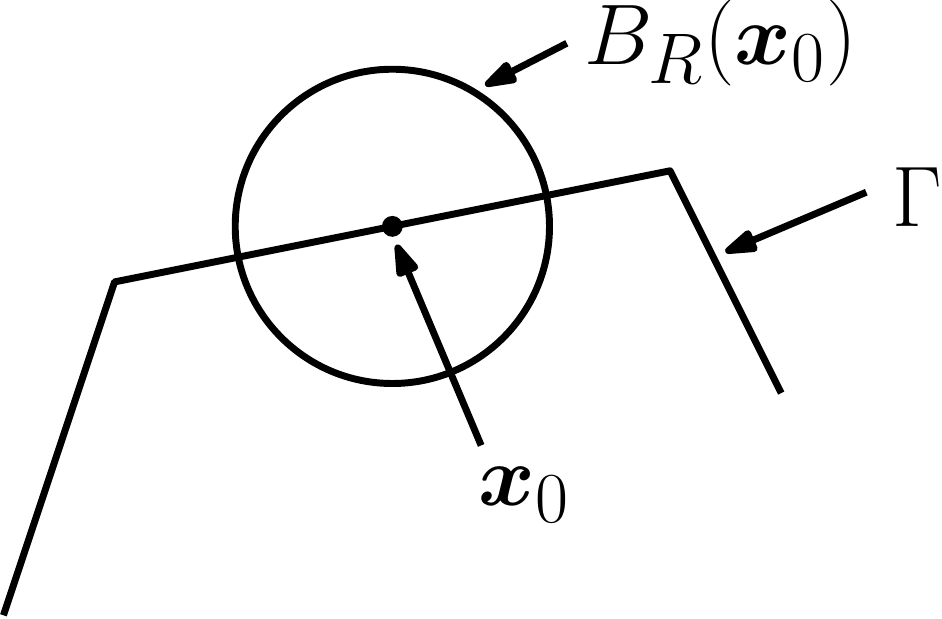}
    \caption{\csentence{The ball $B_R(\bfx_0) \subset \Gamma \cap U$.}}
    \label{fig:ball}
   \end{figure}

  Now, let $\gamma = \Gamma\cap B_R(\bfx_0)$ and note that on
  $\gamma$, both the interface normal $\bfn$ and the jump $[\chi]$ are
  constant. (In fact,
  $[\chi] = \chi_1 - \chi_0 = - (|\Omega_1|^{-1} + |\Omega_2|^{-1})$
  is constant on the entire interface $\Gamma$.)

  To construct the test function $\bfw_c \in Q_c$, we now let
  $\varphi$ be a smooth nonnegative function compactly supported on
  $B_R(\bfx_0)$ and take $\bfv(x) = c \varphi(x) \bfn [\chi]$, where
  again we note that both $\bfn$ and $\chi$ are constants. The
  constant $c$ is chosen such that
  \begin{equation}\label{eq:leminfsupconstantsa}
    (\langle \bfn \cdot \bfv \rangle,[\chi])_\gamma = - \| \chi \|_{\Omega}^2.
  \end{equation}
  Integrating by parts and noting that $\chi$ is constant on each
  subdomain $\Omega_i$, $i = 0, 1$, it follows that this construction
  of $\bfv$ leads to the identity
  \begin{equation}
    b_h(\bfv,\chi)
    = - (\langle \bfn \cdot \bfv\rangle,[\chi])_{\gamma}
    = \| \chi \|_{\Omega}^2.
 \end{equation}
  Now, let
  $\bfw = \bfpi_h \bfv \in \bfV_h$.  It follows that
  \begin{align}
    b_h(\bfw, \chi)
    &=
    b_h(\bfv, \chi) + b_h(\bfw - \bfv, \chi) \\
    &=
    \| \chi \|^2_\Omega -  (\langle  \bfn \cdot (\bfw  - \bfv) \rangle,[\chi])_{\gamma } \\
    &= \| \chi \|^2_\Omega
    - c (\langle \pi_h\varphi - \varphi \rangle,[\chi]^2)_{\gamma} \\
    &\geq \| \chi \|^2_\Omega - cCh \| \chi \|^2_\Omega \\
    &\gtrsim \| \chi \|^2_\Omega.
  \end{align}
  The last inequality holds for all $h\in (0,\bar{h}]$ with $\bar{h}$ sufficiently
  small. (Note that the constants $c$ and $C$ do not depend on $q$.)
  The first inequality follows by noting that
  \begin{align}
    |(\langle \pi_h\varphi - \varphi \rangle,[\chi]^2)_{\gamma}|
    & \leq \|\langle \pi_h\varphi - \varphi \rangle \|_{\gamma} \| [\chi]^2 \|_{\gamma}
    \\
    &\lesssim \| \pi_h \varphi - \varphi \|_{B_R(\bfx_0) }^{1/2}  \| \pi_h \varphi - \varphi \|_{H^1(B_R(\bfx_0))}^{1/2}
    \|\chi\|^2_\Omega
    \\
    &\lesssim h \left( \| \nabla \varphi \|_{B_R(\bfx_0)} + \| \Delta \varphi \|_{B_R(\bfx_0)} \right) \|\chi\|^2_\Omega
    \\
    \label{eq:chiremainder}
    &\lesssim h \|\chi\|^2_\Omega.
  \end{align}
  Here we have used the Cauchy-Schwarz inequality, a trace inequality
  on $B_R(\bfx_0)$, an inequality of the type $ab\lesssim a^2 + b^2$,
  the interpolation estimate (\ref{eq:interpol}) and the definitions
  of $\chi$ \eqref{eq:chi} and $\varphi$. We also note that the
  estimate $\|[\chi]^2\|_{\gamma} \lesssim \|\chi\|^2_{\Omega}$
  follows since $\chi$ is (piecewise) constant.

  Finally, since $q\in Q_c$, we may write $q = c_1 \chi$ for some
  $c_1>0$. (If $c_1 < 0$, we may redefine $\chi$.) Taking
  $\bfw_c = c_2 \bfw$, where $c_2>0$ is chosen such that
  $\tn \bfw_c \tn_h = \| q \|_h$, we have
  \begin{align}
    b_h(\bfw_c,q) &= c_1 c_2 b_h(\bfw,\chi) \\
    &\gtrsim c_1 c_2 \|\chi\|^2_\Omega \\
    &= \frac{c_2}{c_1} \|q\|^2_h \\
    &\gtrsim \|q \|^2_h,
  \end{align}
  since $c_1 = \|q\|_h/\|\chi \|_\Omega$ and $c_2 = \|q\|_h / \tn \bfw
  \tn_h$ and thus $c_1/c_2 = \tn \bfw \tn_h / \| \chi\|_\Omega \sim
  1$.
\end{proof}


\begin{lem}[Inf-sup for $Q_0$]
  \label{lem:infsupsubdomain}
  For each $q \in Q_{0}$ there exists a $\bfw \in \bfW_{h,0} \subset
  \bfV_{h,0}$ with $\|D\bfw \|_{\omega_{h,0}} = \|q -
  \lambda_{\omega_{h,0}}(q)\|_{\omega_{h,0}}$ such that
  \begin{align}
    \| q - \lambda_{\Omega_0}(q) \|^2_{\Omega_{h,0}}
    - h^2\|\nabla q \|^2_{\Omega_{h,0} \setminus \omega_{h,0}}
    &\lesssim
    b_h(\bfw,q),
  \end{align}
  where the bound is uniform w.r.t.~$q$.
\end{lem}
\begin{proof}
  Recall the definitions of $\bfW_{h,0}$ and $\lambda_{\omega_{h,0}}$
  from Assumption \ref{assumptionB}.  We first show that we can change
  the average from $\lambda_{\Omega_0}(q)$ to
  $\lambda_{\omega_{h,0}}(q)$ using the following estimates:
  \begin{align}
    \| q - \lambda_{\Omega_0}(q) \|_{\Omega_{h,0}}
    &\leq
    \| q - \lambda_{\omega_{h,0}}(q) \|_{\Omega_{h,0}}
    + \| \lambda_{\omega_{h,0}}(q)- \lambda_{\Omega_0}(q) \|_{\Omega_{h,0}}
    \\
    &= \| q - \lambda_{\omega_{h,0}}(q) \|_{\Omega_{h,0}}
    + \| \lambda_{\Omega_0}(\lambda_{\omega_{h,0}}(q)- q) \|_{\Omega_{h,0}}
    \\ \label{eq:leminfsupsubdomainsa}
    &\lesssim \| q - \lambda_{\omega_{h,0}}(q) \|_{\Omega_{h,0}},
  \end{align}
  where we first added and subtracted $\lambda_{\omega_{h,0}}(q)$ and
  used the triangle inequality, then used the identity
  $\lambda_{\omega_{h,0}}(q) =
  \lambda_{\Omega_0}(\lambda_{\omega_{h,0}}(q))$, which holds since
  $\lambda_{\Omega_0}$ is an average, and finally we used the
  $L^2(\Omega_{h,0})$ stability $| \lambda_{\Omega_0}(v)| \lesssim
  \|v\|_{\Omega_{h,0}}$ of the average operator.

  Next we have the estimate
  \begin{equation}\label{eq:leminsupsubdomainb}
    \| q \|^2_{\Omega_{h,0}} \lesssim \| q \|^2_{\omega_{h,0}}
    + h^2 \|\nabla q \|_{\Omega_{h,0}\setminus \omega_{h,0}}^2
    \quad \Forall q \in Q_{0},
  \end{equation}
  which follows by first observing that this inverse inequality
  holds:
  \begin{align}\label{eq:inverseelementpair1}
    \| q \|^2_{K_1}
    &\lesssim h^2 \| \nabla q \|^2_{K_1} + h \| q \|_{F_{12}}^2 \\
    \label{eq:inverseelementpair2}
    &\lesssim  h^2 \| \nabla q \|^2_{K_1} + \| q \|^2_{K_2},
  \end{align}
  where $K_1$ and $K_2$ are two neighboring elements sharing the face
  $F_{12}$. Then, starting with $\omega_{h,0}^0 = \omega_{h,0}$, we
  define a sequence of sets $\omega_{h,0}^n$, $n=1,2,\dots$ consisting
  of the union of $\omega_{h,0}^{n-1}$ and all elements $K\subset
  \Omega_{h,0}\setminus \omega_{h,0}^{n-1}$ that share a face with an
  element in $\omega_{h,0}^{n-1}.$ It then follows from
  (\ref{eq:inverseelementpair2}) that
  \begin{equation}\label{eq:sequence}
    \| q \|^2_{\omega_{h,0}^{n}} \lesssim
    \| q \|^2_{\omega_{h,0}^{n-1}} + h^2 \| \nabla q
    \|^2_{\omega_{h,0}^{n} \setminus \omega_{h,0}^{n-1}},  \quad n=1,2,\dots
  \end{equation}
  Using the assumption that $\omega_{h,0}$ is close to $\Omega_0$
  \eqref{assum:omegahclose} together with shape regularity and
  quasi-uniformity of the mesh we conclude that $\omega_{h,0}^{n} =
  \Omega_{h,0}$ for some $n \leq C$ for all $h\in (0,\bar{h}]$ where the
    constant is independent of $h$.  Now (\ref{eq:leminsupsubdomainb})
    follows from a uniformly bounded number of iterations of
    (\ref{eq:sequence}).

    Combining (\ref{eq:leminfsupsubdomainsa}) with
    (\ref{eq:leminsupsubdomainb}), we obtain
    \begin{align}
      \|q - \lambda_{\Omega_{h,0}}(q)\|^2_{\Omega_{h,0}}
      &\lesssim
      \|q - \lambda_{\omega_{h,0}}(q)\|^2_{\Omega_{h,0}}
      \\
      &\lesssim
      \|q - \lambda_{\omega_{h,0}}(q)\|^2_{\omega_{h,0}}
      + h^2\|\nabla q\|^2_{\Omega_{h,0}\setminus \omega_{h,0}}
      \\
      &\lesssim
      b_h(\bfw_0,q)  + h^2\|\nabla q\|^2_{\Omega_{h,0}\setminus \omega_{h,0}},
    \end{align}
    where we used the fact that $\nabla \lambda_{\omega_{h,0}}(q) =
    \bfzero$ and at last the inf-sup condition (\ref{assum:infsup}) to
    choose a $\bfw_0 \in \bfW_{h,0}$ with $\|D\bfw_0 \|_{\omega_{h,0}}
    = \|q - \lambda_{\omega_{h,0}}(q)\|_{\omega_{h,0}}$ such that
    $b_h(\bfw_0,q) = \| q - \lambda_{\omega_{h,0}}(q)
    \|^2_{\omega_{h,0}}$.
\end{proof}


We now combine the inf-sup estimates for $Q_c$ and $Q_0$ to prove an
inf-sup estimate for $Q_h$.
\begin{lem}[Small inf-sup]
  \label{lem:infsupsmall}
  There are constants $c > 0$ and $M > 0$ such that for each
  $q \in Q_h$ there exists a $\bfw \in \bfV_h$ with
  $\tn \bfw \tn_h = \| q \|_h$ such that
  \begin{equation}
    m \|q\|^2_h - C h^2\|\nabla q_0 \|^2_{\Omega_{h,0}
      \setminus \omega_{h,0}} \leq b_h(\bfw ,q),
  \end{equation}
  where the bound is uniform w.r.t.~$q$.
\end{lem}
\begin{proof}
  Take $\bfw_c$ as in Lemma \ref{lem:infsupconstants}, $\bfw_0$ as in
  Lemma \ref{lem:infsupsubdomain} and $\bfw_1 \in
  \bfW_{h,1}$. Consider the test function $\bfw = \delta_1 \bfw_c +
  \bfw_0 + \bfw_1$ where $\delta_1>0$ is a parameter.  By writing $q =
  q_c + q_0 + q_1 \in Q_c \oplus Q_0 \oplus Q_1$, we have
  \begin{align}
    b_h(\delta_1 \bfw_c + \bfw_0 + \bfw_1, q)
    &=
    \delta_1 b_h(\bfw_c, q_c)
    + \delta_1 b_h( \bfw_c,q_0)
    + \delta_1 b_h( \bfw_c,q_1)
    \\ \nonumber
    &\qquad
    + \underbrace{b_h(\bfw_0,q_c)}_{=0}
    + b_h(\bfw_0,q_0)
    + \underbrace{b_h(\bfw_0,q_1)}_{=0}
    \\ \nonumber
    &\qquad
    + \underbrace{b_h(\bfw_1,q_c)}_{=0}
    + \underbrace{b_h(\bfw_1,q_0)}_{=0}
    + b_h(\bfw_1,q_1)
    \\
    &\geq
    \delta_1 m_c \|q_c\|_h^2
    - \delta_1 |b_h(\bfw_c,q_0)|
    - \delta_1 |b_h(\bfw_c,q_1)|
    \\ \nonumber
    &\qquad +  m_0 \|q_0 - \lambda_{\Omega_{0}}(q_0)\|^2_{\Omega_{h,0}}
    -  C h^2\|\nabla q_0 \|^2_{\Omega_{h,0} \setminus \omega_{h,0}}
    \\ \nonumber
    &\qquad + m_1 \|q_1 - \lambda_{\Omega_{1}}(q_1)\|^2_{\Omega_1}
    \\
    &=\bigstar. \label{eq:bigstar}
  \end{align}
  Note that $b_h(\bfw_i,q_c)=0$, $i=0,1$. This follows from
  integration by parts since $q_c\in Q_c$, which is piecewise
  constant, and since $\bfw_i \in \bfW_{h,i}$, which is zero on the
  boundary. The second term and third terms on the right-hand side can
  be estimated as follows
  \begin{align}
    \delta_1 |b_h(\bfw_c,q_i)|
    &=  \delta_1 |b_h(\bfw_c,q_i - \lambda_{\Omega_{h,i}}(q_i))| \\
    &\lesssim \delta_1 \|D \bfw_c\|_{\Omega_{h,i}}
    \|q_i - \lambda_{\Omega_{h,i}}(q_i) \|_{\Omega_{h,i}}
    \\
    &\lesssim \delta_1 \| q_c \|_{\Omega_{h,i}}
    \|q_i - \lambda_{\Omega_{i}}(q_i)\|_{\Omega_{h,i}}
    \\
    &\lesssim \delta_1^2 \delta_2^{-1} \| q_c \|_h^2
    +
      \delta_2 \|q_i - \lambda_{\Omega_{i}}(q_i)\|_{\Omega_{h,i}}^2,
    \label{eq:deltaestimate}
  \end{align}
  where $i=0, 1$ and $\delta_2>0$ is a parameter. Here we have used
  the bound $\| \divv \bfv \| \leq \| D\bfv \|$, the definition of
  $\bfw_c$ from Lemma \ref{lem:infsupconstants} and the inequality $ab
  \leq \epsilon a^2 + (4\epsilon)^{-1} b^2$, which holds for any
  $\epsilon>0$.  Continuing from~\eqref{eq:bigstar}, we
  use~\eqref{eq:deltaestimate} to obtain
  \begin{align}
    \bigstar &\geq
    \delta_1  \left( m_c - C \delta_1 \delta_2^{-1} \right) \|q_c\|_h^2
    + \sum_{i=0}^1 ({m_i} - C \delta_2 )\|q_i-\lambda_{\Omega_{h,i}}(q_i)\|^2_{\Omega_{h,i}}
    \\ \nonumber
    &\qquad - C h^2\|\nabla q_0 \|^2_{\Omega_{h,0} \setminus \omega_{h,0}}
    \\
    &\gtrsim m   \bigg(  \underbrace{ \|q_c\|^2_h
    + \sum_{i=0}^1 \|q^i - \lambda_{\Omega_{i}}(q)\|^2_{\Omega_{h,i}} }_{\displaystyle=\|q\|_h}
    \bigg)
    -  h^2\|\nabla q_0 \|^2_{\Omega_{h,0} \setminus \omega_{h,0}},
  \end{align}
  where we first choose $\delta_2$ sufficiently small and then
  $\delta_1$ sufficiently small to ensure that the two first terms are
  positive.

  Finally, we note that by construction
  \begin{align}
    \tn \bfw \tn_h^2
    &\lesssim \tn \bfw_c \tn_h^2 + \sum_{i=0}^1 \tn \bfw_i \tn^2_h \\
    &= \| q_c \|_h^2 + \sum_{i=0}^1 \| q_i \|^2_{\Omega_{h,i}} \\
    &= \|q\|_h^2
  \end{align}
  and thus $\tn \bfw \tn_h \lesssim \|q\|_h$. The desired result now
  follows by setting $\widetilde{\bfw}=\|q \|_h
  (\bfw/\tn\bfw\tn_h)$, which gives
  \begin{align}
    b_h(\widetilde{\bfw},q)
    &= \frac{\|q \|_h}{\tn\bfw\tn_h} b_h(\bfw,q) \\
    &\gtrsim b_h(\bfw,q).
  \end{align}
\end{proof}

\begin{prop}\label{prop:infsupbig} (Big inf-sup)
  It holds
  \begin{equation}\label{eq:infsupbig}
    \tn (\bfu,p)\tn_h
    \lesssim \sup_{(\bfv,q) \in \bfV_h \times Q_h}
    \frac{A_h((\bfu,p),(\bfv,q))}{\tn (\bfv,q)\tn_h}.
  \end{equation}
\end{prop}
\begin{proof}
Given $p \in Q_h$, take $\bfw \in \bfV_h$ be as in Lemma
\ref{lem:infsupsmall}. First note that for $d_h$ we have the estimate
\begin{align}
  |d_h&((\bfu,p),(\bfw,0))| \\
  &\lesssim h^2 \| \Delta \bfu - \nabla p \|_{\Omega_{h,0} \setminus \omega_{h,0}} \| \Delta \bfw \|_{\Omega_{h,0} \setminus \omega_{h,0}} \nonumber
  \\
  &\lesssim h^2 \left( \| \Delta \bfu \|_{\Omega_{h,0} \setminus \omega_{h,0}} + \| \nabla p \|_{\Omega_{h,0} \setminus \omega_{h,0}} \right) \| \Delta \bfw \|_{\Omega_{h,0} \setminus \omega_{h,0}}
  \\
  &\lesssim \left( \| D \bfu \|_{\Omega_{h,0} \setminus \omega_{h,0}} + h \| \nabla p \|_{\Omega_{h,0} \setminus \omega_{h,0}} \right) \| D \bfw \|_{\Omega_{h,0} \setminus \omega_{h,0}}
  \\
  &\lesssim \delta_2^{-1} \left( \| D \bfu \|^2_{\Omega_{h,0} \setminus \omega_{h,0}} + h^2 \| \nabla p \|^2_{\Omega_{h,0} \setminus \omega_{h,0}} \right) + \delta_2 \| D \bfw \|^2_{\Omega_{h,0} \setminus \omega_{h,0}}
  \\ \label{eq:dhest}
  &\lesssim \delta_2^{-1} \left( \tn \bfu \tn^2_h +  h^2 \| \nabla p \|^2_{\Omega_{h,0} \setminus \omega_{h,0}} \right) + \delta_2 \| p \|_h^2,
\end{align}
where we have used the Cauchy-Schwarz inequality, the triangle
inequality, the inverse estimate \eqref{eq:inverse}, the definition of
the energy norm \eqref{eq:energynorm} and the definition of $\bfw$ in
Lemma \ref{lem:infsupsmall}.

Next for $\delta_1>0$ we have
\begin{align}
  A_h((\bfu,p)&,(\bfu,-p) + \delta_1 (\bfw,0))
  \\
  &= A_h((\bfu,p),(\bfu,-p))  \nonumber
  \\
  &\qquad + \delta_1 \Big(a_h(\bfu,\bfw) + b_h(\bfw,p)
  + d_h((\bfu,p), (\bfw,0)) \Big)  \nonumber
  \\
  &\gtrsim \tn \bfu \tn_h^2 +
  h^2 \| \nabla p \|^2_{\Omega_{h,0} \setminus \omega_{h,0}}
  \\
  &\qquad  - \delta_1 \left( \delta_2^{-1}\tn \bfu \tn_h^2
  + \delta_2 \| p \|_h^2 \right)  \nonumber
  \\
  &\qquad + \delta_1 \left( \|p\|^2_h -  h^2\|\nabla p \|^2_{\Omega_{h,0} \setminus \omega_{h,0}} \right)  \nonumber
  \\
  &\qquad -\delta_1 \delta_2^{-1} \left( \tn \bfu \tn^2_h +  h^2 \| \nabla p \|^2_{\Omega_{h,0} \setminus \omega_{h,0}} \right) - \delta_1 \delta_2 \| p \|_h^2  \nonumber
  \\
  &\gtrsim \left(1 - C \delta_1 \delta_2^{-1}\right) \tn \bfu \tn_h^2 +
  \left(1- C \delta_1\delta_2^{-1}\right) h^2\| \nabla p \|^2_{\Omega_{h,0} \setminus \omega_{h,0}}
  \\
  &\qquad
  + \delta_1\left(1-C \delta_2\right) \| p \|_h^2, \nonumber
\end{align}
where we have used Lemmas \ref{lem:infsuptestu}, \ref{lem:ahcont},
\ref{lem:infsupsmall} as well as \eqref{eq:dhest}.
Choosing first
$\delta_2$ sufficiently small and then $\delta_1$ sufficiently small, we
arrive at the estimate
  \begin{align}
    \tn (\bfu,p)\tn_h^2
    &= \tn \bfu \tn_h^2 + \| p \|_h^2 \\
    &\lesssim A_h((\bfu,p),(\bfu,-p) + \delta_1(\bfw,0))
  \end{align}
  We now note that
  \begin{align}
    \tn (\bfu + \delta_1 \bfw,-p)\tn_h^2
    &=
    \tn \bfu + \delta_1 \bfw \tn_h^2 + \| p \|_h^2 \\
    &\leq
    \tn \bfu \tn_h^2 + \delta_1 \tn \bfw \tn_h^2 + \| p \|_h^2 \\
    &\lesssim
    \tn \bfu  \tn_h^2 + \| p \|_h^2 \\
    &=
    \tn (\bfu,p)\tn_h^2.
  \end{align}
  and thus the desired estimate \eqref{eq:infsupbig} follows since
  \begin{align}
    \tn (\bfu,p)\tn_h
    &\lesssim
    \frac{A_h((\bfu,p),(\bfu + \delta_1 w,-p))\tn_h}{\tn (u, p) \tn_h} \\
    &\lesssim
    \frac{A_h((\bfu,p),(\bfu + \delta_1 w,-p))}{\tn (\bfu + \delta_1\bfw,-p)\tn_h}.
  \end{align}

\end{proof}

\subsubsection{\emph{A~priori} error estimate}

In this section we use the approximation properties of the finite
element spaces to show that the proposed method is optimal.

\begin{thm}
  It holds
  \begin{equation}
    \tn (\bfu,p) - (\bfu_h,p_h) \tn_h
    \lesssim
    h^k(\| \bfu \|_{H^{k+1}(\Omega)} + \| p \|_{H^k(\Omega)} ).
  \end{equation}
\end{thm}
\begin{proof}
  By the triangle inequality we have
  \begin{align}
    \tn (\bfu,p) - (\bfu_h,p_h) \tn_h
    &\lesssim \tn (\bfu,p) - (\bfpi_h \bfu, \pi_h p) \tn_h
    + \tn (\bfpi_h \bfu, \pi_h p) - (\bfu_h, p_h) \tn_h.
  \end{align}
  From the approximation property o\eqref{eq:interpolu}, we obtain an
  optimal estimate of the first term. To show an optimal estimate for
  the second term we recall the big inf-sup estimate Proposition
  \ref{prop:infsupbig}
  \begin{align}
    \tn (\bfpi_h \bfu, \pi_h p) - (\bfu_h, p_h) \tn_h &\lesssim A_h((\bfpi_h \bfu,\pi_h p)-(\bfu_h, p_h),(\bfv_h,q_h)) \\
    &= A_h((\bfpi_h \bfu,\pi_h p)-(\bfu, p),(\bfv_h,q_h)),
  \end{align}
  where we have used a pair $(\bfv_h,q_h)$ such that $\tn (\bfv_h,q_h)
  \tn_h \lesssim 1$ in the inequality and the Galerkin orthogonality
  \eqref{eq:galerkinorth} to obtain the equality. The terms in $A_h$
  \eqref{eq:Ah} may now be estimated individually. The optimal
  estimate for $a_h$ \eqref{eq:ah} follows immediately from continuity
  \eqref{eq:ahcont}. For $b_h(\bfu-\bfpi_h \bfu,q_h)$ \eqref{eq:bh} we
  have
  \begin{align}
    |b_h&(\bfu-\bfpi_h \bfu,q_h)|
    \\
    &\lesssim \left( \sum_{i=0}^1 \|  \divv (\bfu-\bfpi_h \bfu) \|^2_{\Omega_i} \| q_h \|^2_{\Omega_i}
    + \| [ \bfn \cdot (\bfu-\bfpi_h \bfu)] \|^2_\Gamma \| \langle q_h \rangle \|^2_\Gamma \right)^{1/2} \nonumber
    \\
    &\lesssim \left( \| D(\bfu-\bfpi_h \bfu) \|^2_{\Omega_0 \cup \Omega_1} \|q_h\|^2_{\Omega_0 \cup \Omega_1} + h \left( h^{-1} \| [ \bfu-\bfpi_h \bfu] \|^2_\Gamma \right) \|  q_h  \|^2_\Gamma \right)^{1/2}
    \\
    &\lesssim \left( \tn \bfu-\bfpi_h \bfu \tn^2_h \|q_h\|^2_h + h  \tn \bfu-\bfpi_h \bfu \tn^2_h  \|  q_h  \|^2_\Gamma \right)^{1/2},
  \end{align}
  where we have used the Cauchy-Schwarz inequality in the first two
  inequalities and the definition of the energy norm
  \eqref{eq:energynorm} in the last inequality. Using a similar
  argument we obtain the following estimate for $b_h(\bfv_h, p-\pi_h
  p)$:
  \begin{align}
    |b_h(\bfv_h, p-\pi_h p)| &\lesssim \tn \bfv_h \tn_h \| p-\pi_h p \|_h
  \end{align}
  (see \cite{MassingStokes}). Finally, we estimate $d_h$ to obtain
  \begin{align}
    d_h&((\bfu - \bfpi_h \bfu, p-\pi_h p),(\bfv_h,q_h))
    \\
    &\lesssim h^2 \left( \| \Delta (\bfu - \bfpi_h \bfu) \|^2_{\Omega_{h,0} \setminus \omega_{h,0}} + \| \nabla (p-\pi_h p) \|^2_{\Omega_{h,0} \setminus \omega_{h,0}} \right)^{1/2}  \nonumber
    \\
    &\qquad \times \left( \| \Delta \bfv_h \|^2_{\Omega_{h,0} \setminus \omega_{h,0}} + \| q_h \|^2_{\Omega_{h,0} \setminus \omega_{h,0}} \right)^{1/2} \nonumber
    \\
    &\lesssim  \left(  \| D(\bfu - \bfpi_h \bfu) \|^2_{\Omega_{h,0} \setminus \omega_{h,0}} + \| p-\pi_h p \|^2_{\Omega_{h,0} \setminus \omega_{h,0}} \right)^{1/2}
    \\
    &\qquad \times \left( \| D \bfv_h \|^2_{\Omega_{h,0} \setminus \omega_{h,0}} + \| q_h \|^2_{\Omega_{h,0} \setminus \omega_{h,0}} \right)^{1/2} \nonumber
    \\
    &\lesssim \left( \tn \bfu - \bfpi_h \bfu \tn^2_h + \| p-\pi_h p \|^2_h \right)^{1/2} \left( \tn \bfv_h \tn^2_h + \| q_h \|^2_h \right)^{1/2}
    \\
    &= \left( \tn \bfu - \bfpi_h \bfu \tn^2_h + \| p-\pi_h p \|^2_h \right)^{1/2} \tn (\bfv_h,q_h) \tn_h,
  \end{align}
  where we have used the Cauchy-Schwarz inequality, the triangle
  inequality, the inverse estimate \eqref{eq:inverse}, the definition
  of the energy norm \eqref{eq:energynorm} and at last the definition
  of the full triple norm \eqref{eq:triplenorm}. The \emph{a~priori}
  estimate now follows from the interpolation estimates
  \eqref{eq:interpol} and \eqref{eq:interpolu}
\end{proof}

\section{Results and discussion}

\subsection{Numerical results}

To illustrate the proposed method, we here present convergence tests
in 2D and 3D as well as a more challenging problem simulating flow
around a 3D propeller. The numerical results are performed using
FEniCS \cite{LoggMardalEtAl2011,LoggWells2009a}, which is a collection
of free software for automated, efficient solution of differential
equations. The algorithms used in this work are implemented as part of
the ``multimesh'' functionality present in the development version of
FEniCS and will be part of the upcoming release of FEniCS 1.6 in 2015.


\subsubsection{Convergence test}

As a first test case, we consider Stokes flow in the domain
$\Omega = [0,1]^d$, $d=2, 3$, with homogeneous Dirichlet boundary
conditions for the velocity (no-slip) on the boundary. For $d = 2$,
the exact solution is given by
\begin{align}
  \bfu(x, y) &= 2 \pi \sin(\pi x) \sin(\pi y) \cdot
  ( \cos(\pi y) \sin(\pi x) , -\cos(\pi x) \sin(\pi y) ), \\
  p(x, y) &= \sin(2 \pi x) \sin(2 \pi y),
\end{align}
with corresponding right-hand side
\begin{equation}
  \bff(x,y)
  = 2 \pi
  \left(
  \begin{matrix}
    \sin(2 \pi y) (\cos(2 \pi x) - 2 \pi^2 \cos(2 \pi x) + \pi^2) \\
    \sin(2 \pi x) (\cos(2 \pi y) + 2 \pi^2 \cos(2 \pi y) - \pi^2).
  \end{matrix}
  \right)
\end{equation}
For $d=3$, the exact solution is
\begin{align}
  \bfu(x, y, z) &= \sin(\pi y)\sin(\pi z) \cdot
  (1, -\sin(\pi y)\cos(\pi z), \sin(\pi z)\cos(\pi y)), \\
  p(x, y, z) &= \pi\cos(\pi x),
\end{align}
with corresponding right-hand side
\begin{equation}
  \bff(x, y, z)
  = \pi^2
  \left(
  \begin{matrix}
    \sin(\pi y)\sin(\pi z) - \sin(\pi x) \\
    \sin(2\pi z)(2\cos(2\pi y) - 1) \\
    \sin(2\pi y)(1 - 2\cos(2\pi z)
  \end{matrix}
  \right).
\end{equation}
In both cases, the velocity field is divergence free and the
right-hand side has been chosen to match the given exact solutions. We
let the overlapping domain $\Omega_1$ be a $d-$dimensional cube
centered in the center of $\Omega$ with side length $0.246246$ rotated
$37$\textdegree\ along the $z$-axis. For $d=3$, $\Omega_1$ is rotated
the same angle along the $y$-axis as well. The domains $\Omega_i$ are
illustrated in Figure \ref{fig:cube_domains}.

The discrete spaces are $P_{k}$--$P_{k-1}$ Taylor--Hood finite
element spaces with continuous piecewise vector-valued polynomials of
degree $k$ discretizing the velocity and discontinuous scalar
polynomials of degree $l = k - 1$ discretizing the pressure. These spaces
are inf-sup stable on the uncut elements of the background mesh
discretizing $\Omega_0$ and on the whole of $\Omega_1$ and therefore
satisfy Assumption~\ref{assumptionB}.

\begin{figure}
  \includegraphics[width=0.65\textwidth]{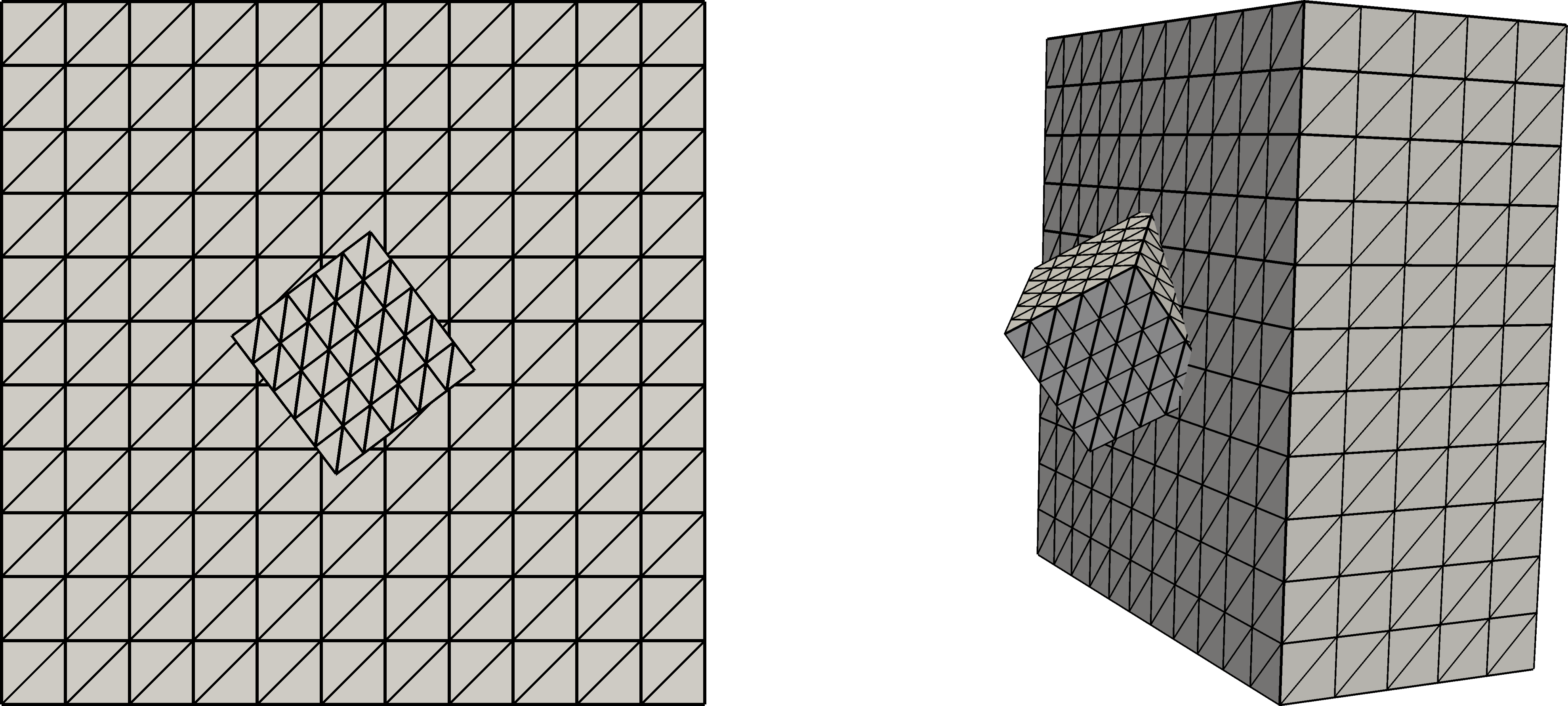}
  \caption{\csentence{Location of the overlapping domain in the
      background mesh.} The domain $\Omega_1$ is placed in the center
    of $\Omega$ and rotated along the $z$-axis in 2D (left) and along
    the $y$- and $z$-axes in 3D (right). }
  \label{fig:cube_domains}
\end{figure}

Figures \ref{fig:convergence_2D} and \ref{fig:convergence_3D} show the
convergence of the error in the $H^1_0$- and $L_2$-norms in 2D and 3D
respectively.  Optimal order of convergence is obtained, although
limited computer memory resources prevented a study for higher degrees
than $k = 3$ in 3D.
In the convergence plots,
results for small mesh sizes, roughly corresponding to errors below
$10^{-7}$ have been removed because errors could not be reliably
estimated due to numerical round-off errors in the numerical
integration close to the cut cell boundary.

\begin{figure}
  \includegraphics[width=0.9\textwidth]{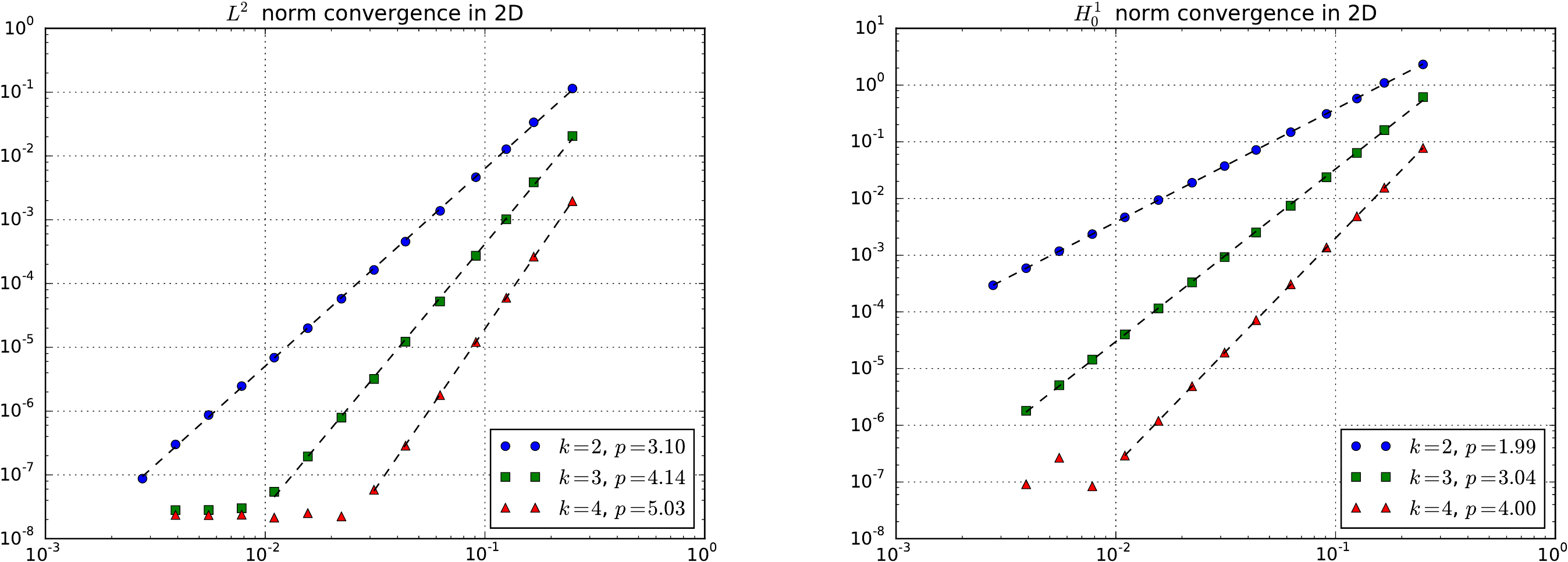}
  \caption{\csentence{Convergence results, 2D.} A rotated square is
    embedded in the unit square background mesh. Results in $L^2$
    (left) and $H^1_0$ (right) norms. }
    \label{fig:convergence_2D}
\end{figure}

\begin{figure}
  \includegraphics[width=0.9\textwidth]{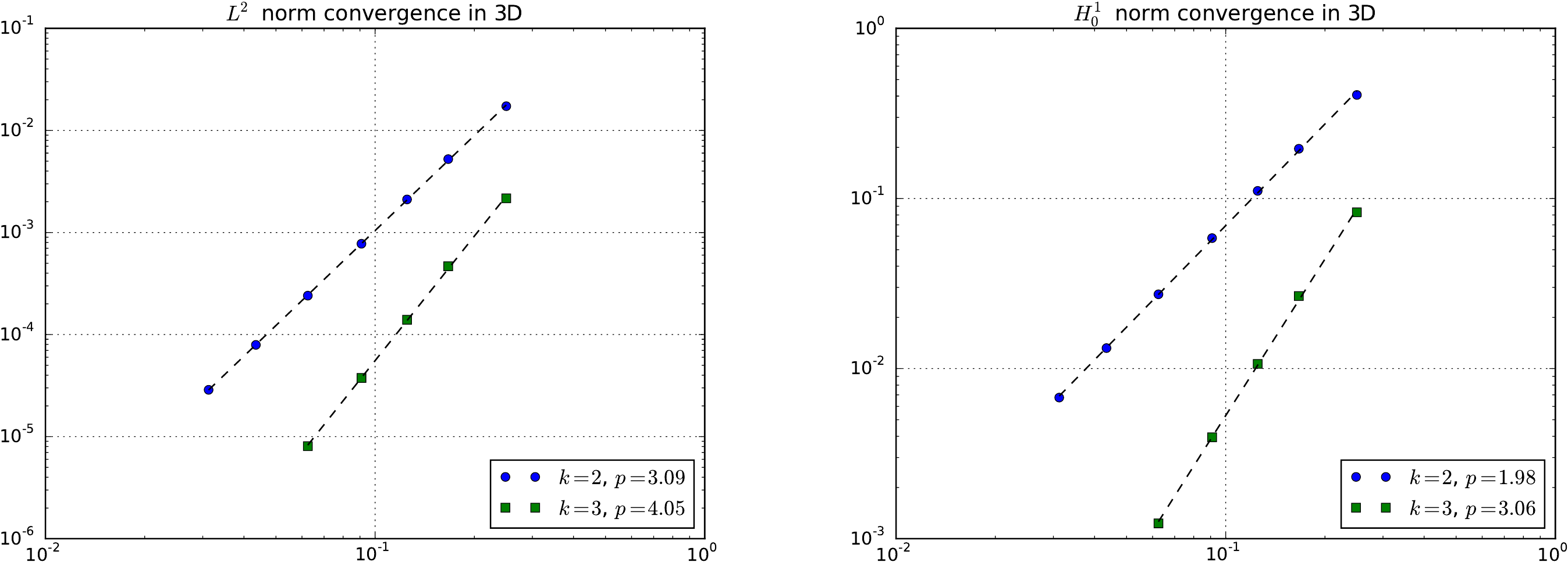}
  \caption{\csentence{Convergence results, 3D.} A rotated cube is
    embedded in the unit cube background mesh. Results in $L^2$
    (left) and $H^1_0$ (right) norms. }
    \label{fig:convergence_3D}
\end{figure}

\subsubsection{Flow around a propeller}

To illustrate the method on a complex geometry we create a propeller
using the CSG tools of the FEniCS component mshr \cite{mshr}, see
Figure \ref{fig:propeller-mesh} (top left). The lengths of the blades
are approximately $0.5$. Then we construct a mesh of the domain
outside the propeller, but inside the unit sphere. This is illustrated
in Figure \ref{fig:propeller-mesh} (top right). The mesh is
constructed using TetGen \cite{tetgen} and is body-fitted to the
propeller. To simulate the flow around the propeller, the mesh is
placed in a background mesh of dimensions $[-2,2]^3$, where we have
removed the elements with all nodes inside a sphere of radius $0.9$,
see Figure \ref{fig:propeller-mesh} (bottom).

\begin{figure}
  \includegraphics[height=0.4\textheight]{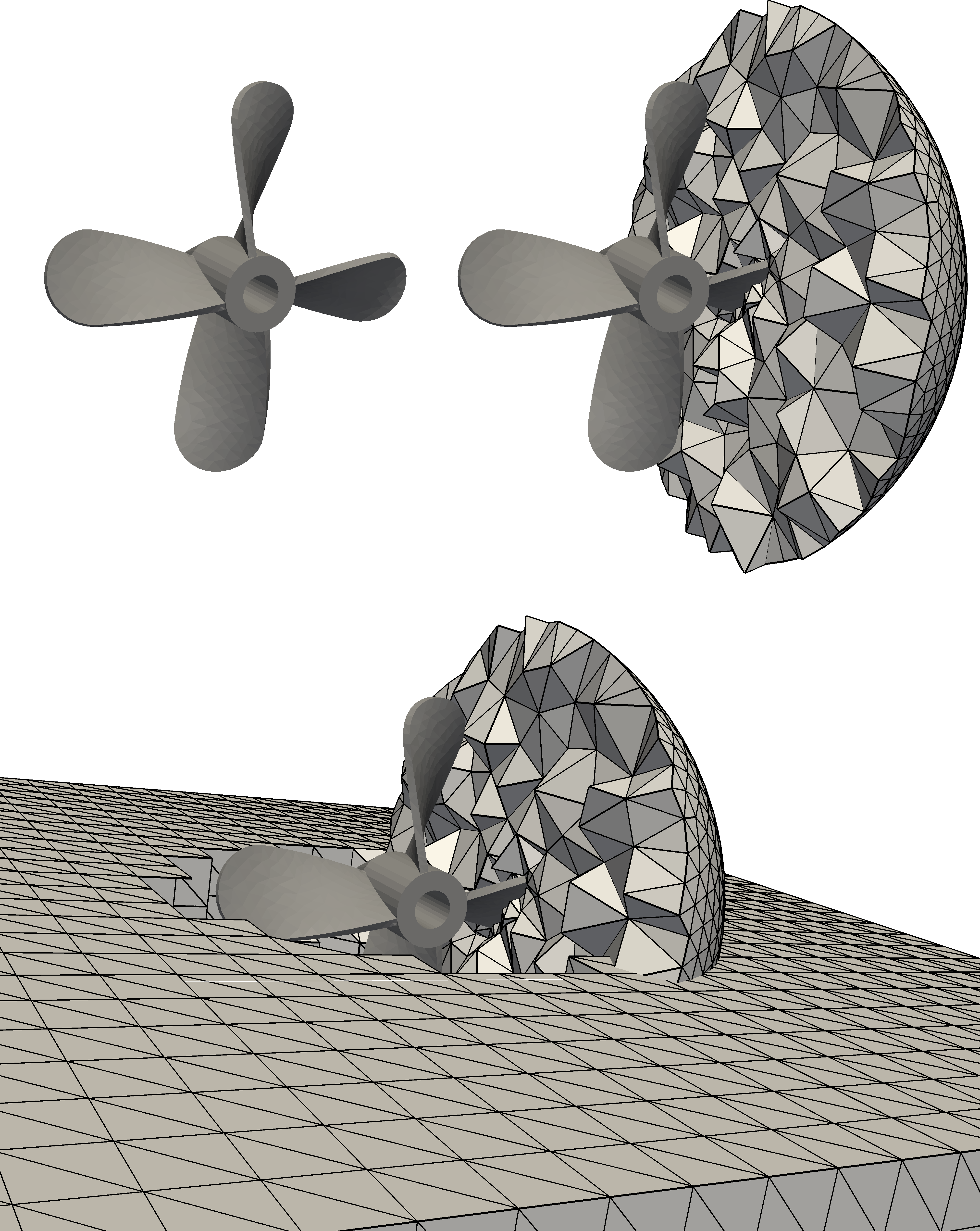}
  \caption{\csentence{Propeller geometry and meshes.} Propeller
    geometry (top left) and body-fitted mesh (top right). Non
    body-fitted background mesh and propeller (bottom).}
  \label{fig:propeller-mesh}
\end{figure}

The simulation is setup with the inflow condition
$\bfu(x,y,z) = (0,0,\sin(\pi(x+2)/4) \sin(\pi(y+2)/4))$ at $z=-2$, the
outflow condition $p=0$ at $z=2$ and $\bfu(x,y,z)=0$ on all other
boundaries, including the boundary of the propeller. The resulting
velocity field using degree $k = 2$ is shown in Figure
\ref{fig:propeller-flow}. Note the continuity of the streamlines of
the velocity going from the finite element space defined on the
background mesh to the finite element space defined on the overlapping
mesh surrounding the propeller.

\begin{figure}
  \includegraphics[width=0.75\textwidth]{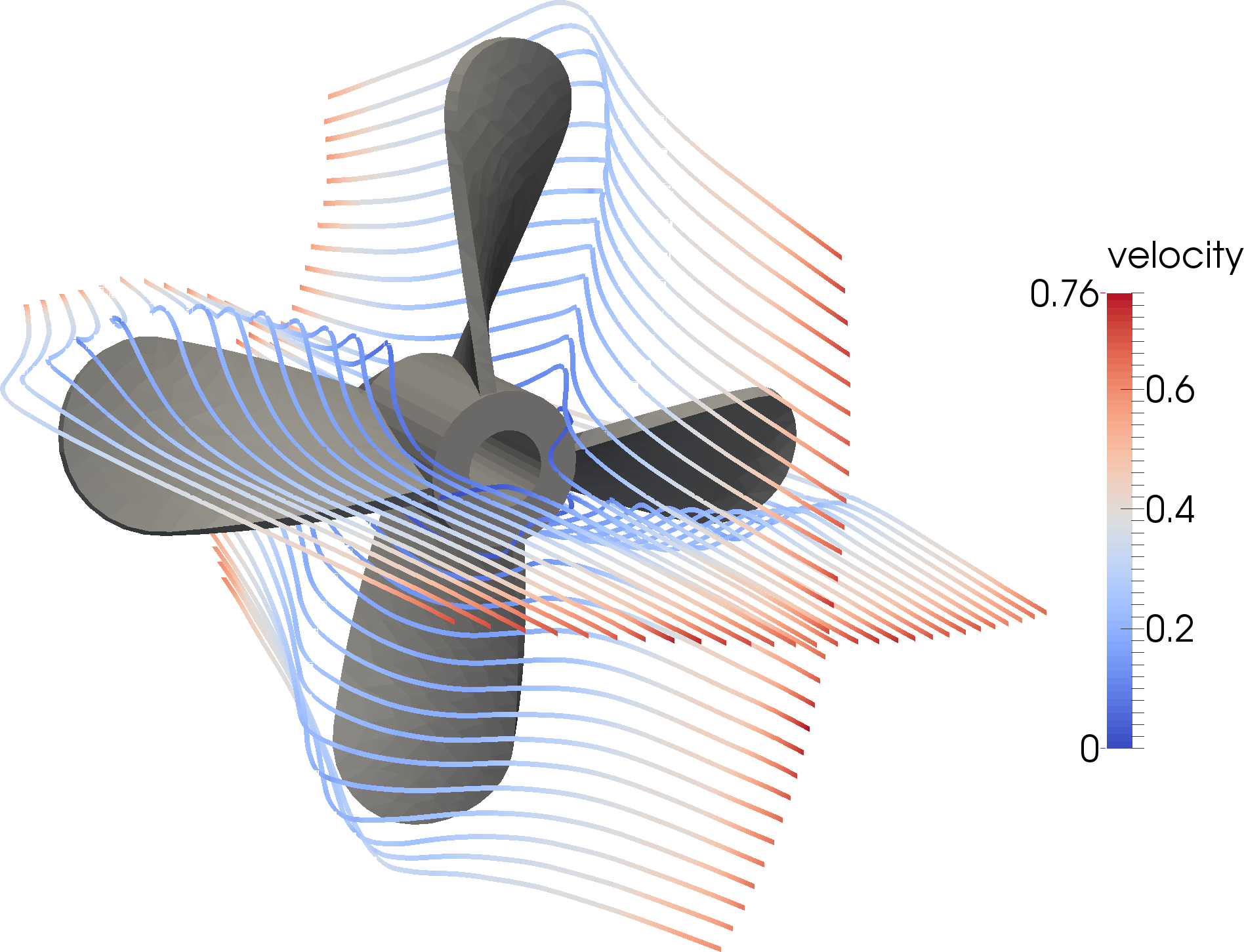}
  \caption{\csentence{Flow around propeller.} Colors indicate speed.}
  \label{fig:propeller-flow}
\end{figure}

\section{Conclusions}

The finite element formulation for discretization of the Stokes
problem presented has been demonstrated to have optimal order
convergence, first by an \emph{a~priori} error estimates and then
confirmed by numerical results. The finite element formulation studied
in this work allows inf-sup stable spaces for the Stokes problem to be
stitched together from multiple non-matching and intersecting meshes
to form a global inf-sup stable space. The method has several
practical applications and one such prime example is the
discretization of flow around complex objects. Future work includes
the extension to time-dependent problems and to fluid--structure
interaction.

\bibliographystyle{amsplain}
\bibliography{bibliography}

\providecommand{\bysame}{\leavevmode\hbox to3em{\hrulefill}\thinspace}
\providecommand{\MR}{\relax\ifhmode\unskip\space\fi MR }
\providecommand{\MRhref}[2]{%
  \href{http://www.ams.org/mathscinet-getitem?mr=#1}{#2}
}
\providecommand{\href}[2]{#2}
\begin{thebibliography}{10}

\bibitem{mini}
D.N. Arnold, F.~Brezzi, and M.~Fortin, \emph{{A stable finite element for the
  Stokes equations}}, Calcolo \textbf{21} (1984), no.~4, 337--344.

\bibitem{chimera}
J.~A. Benek, P.~G. Buning, and J.~L. Steger, \emph{{A 3-D chimera grid
  embedding technique}}, Tech. Report 85-1523, AIAA, 1985.

\bibitem{boffi}
D.~Boffi, F.~Brezzi, and M.~Fortin, \emph{Mixed finite element methods and
  applications}, Springer-Verlag, Berlin Heidelberg, 2013.

\bibitem{braess}
D.~Braess, \emph{{Finite elements: Theory, fast solvers, and applications in
  solid mechanics}}, Cambridge University Press, Cambridge, 2007.

\bibitem{BreSco08}
S.~C. Brenner and L.~R. Scott, \emph{{The Mathematical Theory of Finite Element
  Methods}}, Springer, New York, 2008.

\bibitem{BreFor}
F.~Brezzi and M.~Fortin, \emph{Mixed and hybrid finite element methods},
  Springer-Verlag, New York, 1991.

\bibitem{henshaw}
G.~Chesshire and W.~D. Henshaw, \emph{Composite overlapping meshes for the
  solution of partial differential equations}, J. Comput. Phys. \textbf{90}
  (1990), no.~1, 1--64.

\bibitem{CR1973}
M.~Crouzeix and P.-A. Raviart, \emph{Conforming and nonconforming finite
  element methods for solving the stationary {S}tokes equations {I}}, RAIRO
  Anal. Numer. \textbf{7} (1973), no.~3, 33--75.

\bibitem{wall1}
A.~Gerstenberger and W.~A. Wall, \emph{{An eXtended Finite Element
  Method/Lagrange multiplier based approach for fluid–structure
  interaction}}, Comput. Method Appl. M. \textbf{197} (2008), no.~19–20, 1699
  -- 1714.

\bibitem{HanHanLar}
A.~Hansbo, P.~Hansbo, and M.~G. Larson, \emph{A finite element method on
  composite grids based on {N}itsche's method}, ESAIM-Math. Model. Num.
  \textbf{37} (2003), no.~3, 495--514.

\bibitem{ZahediP1isoP2}
P.~Hansbo, M.~G. Larson, and S.~Zahedi, \emph{A cut finite element method for a
  {S}tokes interface problem}, Appl. Numer. Math. \textbf{85} (2014), 90--114.

\bibitem{codina}
G.~Houzeaux and R.~Codina, \emph{{A Chimera method based on a
  Dirichlet/Neumann(Robin) coupling for the Navier-Stokes equations}}, Comput.
  Method Appl. M. \textbf{192} (2003), no.~31–32, 3343 -- 3377.

\bibitem{mshr}
B.~Kehlet, \emph{{mshr: Mesh generation component of FEniCS}},
  \url{https://bitbucket.org/benjamik/mshr}, Accessed: 2015-01-08.

\bibitem{LoggMardalEtAl2011}
A.~Logg, K-A. Mardal, G.~N. Wells, et~al., \emph{Automated solution of
  differential equations by the finite element method}, Springer-Verlag, Berlin
  Heidelberg, 2012.

\bibitem{LoggWells2009a}
A.~Logg and G.~N. Wells, \emph{{DOLFIN}: Automated finite element computing},
  ACM Transactions on Mathematical Software \textbf{37} (2010), no.~2.

\bibitem{MasLarLog13}
A.~Massing, M.~G. Larson, and A.~Logg, \emph{Efficient implementation of finite
  element methods on nonmatching and overlapping meshes in three dimensions},
  SIAM J. Sci. Comput. \textbf{35} (2013), no.~1, 23--47.

\bibitem{MassingStokes}
A.~Massing, M.~G. Larson, A.~Logg, and M.~E. Rognes, \emph{A stabilized
  {N}itsche overlapping mesh method for the {S}tokes problem}, Numer. Math.
  \textbf{128} (2014), no.~1, 73--101.

\bibitem{Nitsche71}
J.~Nitsche, \emph{{\"U}ber ein {V}ariationsprinzip zur {L}\"osung von
  {D}irichlet-{P}roblemen bei {V}erwendung von {T}eilr\"aumen, die keinen
  {R}andbedingungen unterworfen sind}, Abh. Math. Sem. Hamburg \textbf{36}
  (1971), 9--15.

\bibitem{ScottZhang}
L.~R. Scott and S.~Zhang, \emph{Finite element interpolation of nonsmooth
  functions satisfying boundary conditions}, Math. Comp. \textbf{54} (1990),
  483--493.

\bibitem{wall2}
S.~Shahmiri, A.~Gerstenberger, and W.~A. Wall, \emph{{An XFEM-based embedding
  mesh technique for incompressible viscous flows}}, Int. J. Numer. Meth. Fl.
  \textbf{65} (2011), no.~1-3, 166--190.

\bibitem{tetgen}
H.~Si, \emph{{TetGen: A Quality Tetrahedral Mesh Generator and
  Three-Dimensional Delaunay Triangulator}}, \url{http://www.tetgen.org},
  Accessed: 2015-01-08.

\bibitem{verfurth}
R.~Verf\"urth, \emph{Error estimates for a mixed finite element approximation
  of the {S}tokes equation}, RAIRO Anal. Numer. \textbf{18} (1984), 175--182.

\end{thebibliography}

\end{document}